\documentclass[leqno,twoside,11pt]{amsart}

\usepackage{latexsym,esint,url}

\theoremstyle{plain}

\newtheorem{theorem}{Theorem}[section]

\newtheorem{lemma}[theorem]{Lemma}
\newtheorem{proposition}[theorem]{Proposition}
\newtheorem{definition}[theorem]{Definition}
\theoremstyle{definition}

\newtheorem{remark}[theorem]{Remark}

\numberwithin{equation}{section}
\numberwithin{figure}{section}

\renewcommand{\d}{\partial}
\newcommand{\Tr}{\mbox{Trace }}
\newcommand{\e}{\varepsilon}
\newcommand{\RR}{\mathbb R}
\begin{document}

\title[Matching properties of infinitesimal isometries for developable shells]
{Infinitesimal isometries on developable surfaces \\
and asymptotic theories for thin developable shells}
\author{Peter Hornung, Marta Lewicka and Mohammad Reza Pakzad}
\address{Peter Hornung, Department of Mathematical Sciences, University of Bath, Bath BA2 7AY, UK}
\address{Marta Lewicka,  Rutgers University, Department of Mathematics,
110 Frelinghuysen Rd., Piscataway, NJ 08854-8019}
\address{Mohammad Reza Pakzad, University of Pittsburgh, Department of Mathematics,
139 University Place, Pittsburgh, PA 15260}
\email{p.hornung@bath.ac.uk, lewicka@math.rutgers.edu, pakzad@pitt.edu}

\begin{abstract}
We perform a detailed analysis of first order Sobolev-regular infinitesimal isometries on
developable surfaces without affine regions.
We prove that given enough regularity of the surface, any first order infinitesimal
isometry can be matched to an infinitesimal isometry  of an
arbitrarily high order. We discuss the implications
of this result for the elasticity of thin developable shells. \\

\noindent {\sc Keywords.} developable surfaces, shell theories, nonlinear elasticity, calculus of variations

\noindent {\sc Mathematics Subject Classification.} 74K20, 74B20
\end{abstract}


\maketitle

\section{Introduction}

The derivation of asymptotic theories for thin elastic films has been
a longstanding problem in the mathematical theory of elasticity \cite{ciarbook}. 
Recently, various lower dimensional theories have been  rigorously
derived from the nonlinear $3$ dimensional model, 
through $\Gamma$-convergence \cite{dalmaso} methods.  
Consequently, what seemed to be competing and contradictory theories for elastica (rods,
plates, shells, etc) are now revealed to be each
valid in their own specific range of parameters such as material elastic
constants, boundary conditions and force magnitudes \cite{LR1, FJMgeo,
  FJMhier, Mor-Mue, CMM}. In this line, Friesecke, James and M\"uller
gave a detailed description of the so called hierarchy of plate
theories in \cite{FJMhier}, corresponding to distinct energy scaling
laws in terms  of the plate thickness. Similar results have been established for elastic shells
\cite{LeD-Rao, FJMM_cr, lemopa1, lemopa2, lemopa3}, 
however the description is still far from being complete.

In \cite{LePa2} Lewicka and Pakzad put forward a conjecture regarding
existence of infinitely many small slope  shell theories
each valid for a corresponding range of energy scalings. This
conjecture, based on formal asymptotic expansions, is in
accordance with all therigorously obtained results for plates and shells.
It predicts the form of the 2 dimensional 
limit energy functional, and identifies the space of admissible
deformations as infinitesimal isometries of a given integer order
$N>0$ determined by the magnitude of the elastic energy.  Hence, 
the influence of shell's geometry on its qualitative response to an
external force, i.e. the shell's rigidity, is reflected in a hierarchy of functional spaces of
isometries (and infinitesimal isometries) arising as constraints of
the derived theories. 

In certain cases, a given $N$th order infinitesimal isometry can be modified by higher order 
corrections to yield an infinitesimal isometry of order $M>N$, a property to which we refer to by 
{\it matching property of infinitesimal isometries}. This feature, 
combined with certain density results for spaces of isometries, causes
the theories corresponding to orders of infinitesimal isometries between $N$ and $M$ 
to collapse all into one and the same 
theory.   Examples of such behavior are observed for plates \cite{FJMhier}, where any 
second order infinitesimal isometry can be matched to an exact
isometry (and hence to an $M$th-order isometry 
for all $M\in \mathbb N$), and for convex shells \cite{lemopa3}, where any first order
infinitesimal isometry enjoys the same property. 
The effects of these observations on the elasticity of  thin films are drastic. 
A plate possesses three types of small-slope theories: the linear theory, 
the von K\'arm\'an theory and the linearized Kirchhoff theory \cite{FJMhier}, whereas the 
only small slope theory for a  convex shell is the 
linear theory \cite{lemopa3}: a convex shell transitions directly from the linear regime
to fully nonlinear bending if the applied forces are adequately 
increased. In other words,  while the von K\'arm\'an theory describes buckling of thin plates,
the equivalent variationally correct theory for elliptic shells is the
purely nonlinear bending.  

\medskip

In this paper, we focus on developable surfaces (without flat regions).
This class includes smooth cylindrical shells which
are ubiquitous in nature and technology over a range of length
scales.  An example of a recently discovered structure is carbon
nanotubes,  i.e. molecular-scale tubes of graphitic carbon with
outstanding rigidity properties \cite{Harris}: they are among the stiffest materials in terms
of the tensile strength  and elastic modulus,  but they easily
buckle  under compressive, torsional or bending stress \cite{Jensen}.
The common approach in studying buckling phenomenon for cylindrical
tubes has been to use the von K\'arm\'an-Donnel equations
\cite{mahavazir, peletier}.
However, as we establish here, the proper sublinear theory for this purpose
is, again,  the purely nonlinear bending theory.
It seems likely that the von K\'arm\'an-Donnel equations could be
rigorously derived and valid in another
scaling limit, e.g. when the radius of the cylinder is very large as the thickness vanishes.

The key ingredient of our analysis is a study of $W^{2,2}$ first
order isometries on developable surfaces, of which we address
regularity, rigidity, density and matching properties. Our results
depend on the regularity of the surface and a certain mild convexity property. 
More precisely, we establish that any $\mathcal{C}^{2N-1,1}$ regular
first order infinitesimal isometry on a developable $\mathcal{C}^{2N,1}$ surface
with a positive lower bound on the mean curvature, can be matched to an $N$th-order infinitesimal 
isometry. Combined with a density result for $W^{2,2}$ first order isometries on such surfaces, 
we prove that the limit theories for the energy scalings of the order lower than 
$h^{2+2/N}$ collapse all into the linear theory. 
Our method is to inductively solve the linearized metric equation sym$\nabla w= B$ on the surface
with suitably chosen right hand sides, 
a process during which we lose regularity: only if the surface is
$\mathcal{C}^\infty$ we can establish the total collapse of all small slope 
The importance of the solvability of sym$\nabla w= B$
has been noted in \cite{sanchez}, 
see also \cite{GSP} and \cite{choi} in regards to relations of
rigidity and elasticity.  
We remark here that a surface, e.g. a regular cylinder, may well be ill-inhibited according 
to the definition by Sanchez-Palencia  in \cite{sanchez} and yet satisfy the adequate
matching properties resulting in the aforementioned collapse.

Our analysis can be generalized to piecewise smooth surfaces which satisfy the 
convexity property as above, on each component. The question of
existence of a developable or non-developable surface 
with no flat regions which shows a different elastic behavior
(i.e. validity of an intermediate theory between the linear one and
the purely nonlinear bending) remains open. 

\medskip

Finally, a word on the developability property of surfaces of
vanishing Gaussian curvature,  from which the term {\it developable} 
is derived, is to the point. For each point on such a
surface there exists a straight  segment passing through it and 
lying on the surface, in a characteristic 
direction. Developable surfaces are also locally identified with
isometric images of domains in $\mathbb R^2$;  this 
last property and, in particular, the developable structure of
isometries of flat domains is heavily exploited in this paper. Such structure was established for 
$\mathcal{C}^2$ isometries in \cite{HartmanNirenberg}, for $\mathcal{C}^1$ isometries with
total zero curvature in \cite[Chapter II]{Po 56}, \cite[Chapter IX]{Po 73} 
and for $W^{2,2}$ isometries in \cite{kirchheimthesis, Pak}.
In \cite{Pak} Pakzad proved that any $W^{2,2}$ isometry on a convex
domain can  be approximated in strong norm by smooth isometries, and in \cite{MuPa2} the boundary 
regularity was discussed. Lately, Hornung systematically represented 
and generalized these results in \cite{Ho0, Ho1, Ho2} whose
terminology we will adapt for the sake of simplicity and
completeness. The above mentioned results  have had other applications 
in nonlinear elasticity \cite{CD08, LM, Ho3}.
 
The paper is organized as follows. In section \ref{prelim}, we introduce and
review preliminary facts about developable  surfaces.
In section \ref{equation}, we study the linearized equation of
isometric immersion of a developable surface and prove
existence of a solution operator with suitable bounds in Theorem
\ref{prop1}. We proceed to study the space of $W^{2,2}$ first order
infinitesimal isometries and prove a compensated regularity and rigidity property of
such mappings in Theorem \ref{prop1-1storder} of section \ref{space}.
We then use these results to prove in Theorem \ref{th_nth-intro} that
given enough regularity of the surface, any first order infinitesimal
isometry can be matched to a higher order one.
Combined with a straightforward density result in Theorem
\ref{th_density-intro},  we are finally able to derive the main
application of this paper in elasticity theory,  namely Theorem
\ref{thlimsup-intro},  which is the counterpart to Theorem
\ref{thliminf-intro} for deducing the $\Gamma$-limit of
3 dimensional nonlinear elasticity for thin developable shells.

\bigskip

\noindent{\bf Acknowledgments.}
P.H. was supported by EPSRC grant EP/F048769/1.
M.L. was partially supported by the NSF grants DMS-0707275 and DMS-0846996,
and by the Polish MN grant N N201 547438.
M.R.P. was partially supported by the University of Pittsburgh grant CRDF-9003034 and by
the NSF grant DMS-0907844.

\section{Developable surfaces}\label{prelim}

Let $\Omega\subset\RR^2$ be a bounded Lipschitz domain,
and let $u\in \mathcal{C}^{2,1}(\Omega, \RR^3)$ be an isometric immersion of $\Omega$ into $\RR^3$:
$$ \d_i u(x)\cdot\d_j u(x) = \delta_{ij} \quad \mbox{ for all
}x\in\Omega. $$
A classical result \cite{HartmanNirenberg} asserts that, away from
affine regions, such $u$ must be developable: the domain $\Omega$ can
be decomposed (up to a controlled remainder)
into finitely many subdomains on which $u$ is affine and finitely
many subdomains on which $u$ admits a line of curvature parametrization, see e.g. Theorem 4
in \cite{Ho1}. In light of this result, it is natural to restrict ourselves to the situation
when $\Omega$ can be covered by a single line of curvature chart.

We now make the above more precise. Let $T > 0$, 
let $\Gamma\in \mathcal{C}^{1,1}( [0, T], \RR^2)$ be an arclength
parametrized curve, 
and let $s^{\pm}\in \mathcal{C}^{0,1}([0,T])$ be positive functions
(a rationale for assuming Lipschitz continuity here can be found in Lemma 2.2 in \cite{Ho3}).
Following the notation in \cite{Ho2}, we set:
$$ N = (\Gamma')^{\perp}\in \mathcal{C}^{0,1}([0,T],\RR^2) \quad
\mbox{ and } \quad 
\kappa = \Gamma''\cdot N\in L^\infty(0,T)$$ the unit normal and the curvature of $\Gamma$.
We introduce the bounded domain:
$$ M_{s^{\pm}} = \{(s, t) : t\in (0, T), ~~
s\in (-s^-(t), s^+(t)) \}, $$
the mapping $\Phi : M_{s^{\pm}} \to \RR^2$ given by:
$$ \Phi(s, t) = \Gamma(t) + sN(t) $$
and the open line segments:
$$ [\Gamma(t)] = \{\Gamma(t) + sN(t) : s\in (-s^-(t), s^+(t))\}. $$
From now on we assume that:
\begin{equation}
\label{admissible}
[\Gamma(t_1)]\cap[\Gamma(t_2)] = \emptyset \quad \mbox{ for all unequal }t_1, t_2\in [0, T].
\end{equation}
This condition will be needed to define a  developable isometry $u$, 
which maps segments $[\Gamma(t)]$ to segments in $\mathbb{R}^3$.

The next lemma is a consequence of Proposition 2 and Proposition 1 in \cite{Ho2}.

\begin{lemma}
\label{lem1}
Given $\kappa_n\in L^2(0,T)$, let  $r = (\gamma', v, n)^T\in
W^{1,2}((0,T), SO(3))$ satisfy the ODE:
\begin{equation}\label{darboux}
 r'=\left( \begin{array}{ccc}
  0 & \kappa & \kappa_n\\
  -\kappa & 0 & 0\\
  -\kappa_n & 0 & 0
 \end{array}\right) r
\end{equation}
with initial value
$r(0) = Id$, where we set $\gamma(t) = \int_0^t\gamma'$. Define the mapping:
\begin{equation}\label{induced-1} 
u : \Phi(M_{s^{\pm}}) \to\RR^3, \qquad
u(\Phi(s,t)) = \gamma(t) + sv(t) \quad \forall (s,t)\in M_{s^{\pm}}.
\end{equation}
Then $u$ is well defined and:
$$ u\in W^{2,2}_{\rm loc}(\Phi(M_{s^{\pm}}), \RR^3) $$
is an isometric immersion.
Moreover, $\nabla u(\Phi)\in \mathcal{C}^0(\bar{M}_{s^{\pm}}, \RR^{3\times 2})$
with:
\begin{equation}\label{gradient}
 \nabla u(\Phi(s,t)) =
\gamma'(t)\otimes\Gamma'(t) + v(t)\otimes N(t) \qquad \forall (s,t)\in
\bar M_{s^{\pm}}.
\end{equation}
and the functions:
$$ a_{ij} =  (\d_1 u\times \d_2 u) \cdot \d_i\d_j u $$
satisfy:
\begin{equation}\label{uijAij}
\d_i\d_j u = a_{ij} (\d_1 u\times \d_2 u) \qquad \forall i,j=1,2,
\end{equation}
\begin{equation}\label{A}
a_{ij}(\Phi(s,t)) =
\frac{\kappa_n(t)}{1-s\kappa(t)}\Gamma_i'(t)\Gamma_j'(t) \quad \mbox{ for
almost every }(s,t) \in M_{s^{\pm}}.
\end{equation}
If, in addition:
\begin{equation}\label{LL^11}
\int_0^T\Bigg(\int_{s^-(t)}^{s^+(t)} \frac{\kappa_n^2(t)}{1 - s\kappa(t)} \mathrm{d}s\Bigg)\mathrm{d}t
< \infty
\end{equation}
then $u\in W^{2,2}(\Phi(M_{s^{\pm}})), \RR^3)$ with:
\begin{equation}\label{energie}
\int_{\Phi(M_{s^{\pm}})} |\nabla^2 u(x)|^2\mathrm{d}x = \int_0^T
\Big(\int_{s^-(t)}^{s^+(t)}\frac{\kappa_n^2(t)}{1-s\kappa(t)}\
\mathrm{d}s\Big)\mathrm{d}t.
\end{equation}
\end{lemma}

In order to keep the hypotheses short, we make the following definition:

\begin{definition}
\label{def0}
A surface $S\subset\RR^3$ is said to be developable of class $\mathcal{C}^{k,1}$ if
there are $\Gamma$, $N$, $s^{\pm}$, $\gamma$, $v$, $n$, $\kappa$, $\kappa_n$, $\Phi$
and $u$ as in Lemma \ref{lem1} such that
\begin{equation}
\label{chart}
\Phi(M_{s^{\pm}}) = \Omega,
\end{equation}
$u\in \mathcal{C}^{k,1}(\overline{\Omega}, \RR^3)$ and $S = u(\Omega)$.
\end{definition}

\begin{remark}
\begin{enumerate}
\item The curve $\gamma$ is a line of curvature of surface $S$.
Hence the mapping $(s, t)\mapsto\gamma(t) + sv(t)$
is a line of curvature parametrization of $S$
and the condition \eqref{chart} is the precise formulation of our
assertion that $\Omega$ is covered by a single line of curvature
chart.
\item The moving frame $r$ is the Darboux frame on the surface $S$ along
$\gamma$. Therefore, (\ref{darboux})  indicates that the
geodesic curvature of $\gamma$ coincides with the curvature $\kappa$ of its
preimage $\Gamma$, and that its geodesic torsion vanishes.  This is
naturally expected as $u$ is an isometry. In the same vein, $\kappa_n$ is the normal
curvature of $\gamma$ on $S$.

\item Of course, $n$ is the unit normal vector to $S$ and
  $-[a_{ij}]$ is the second fundamental form
(expressed in $u$-coordinates).  Equation \eqref{A} shows that
it has rank one or zero, hence the Gauss curvature $\det [a_{ij}]$ is zero.

\item Condition (\ref{LL^11}) implies that
$\kappa_n = 0$ almost everywhere on the set:
$$ I_0 = \Big\{t\in (0,T) : \kappa(t)\in\{1/{s^-(t)},
1/{s^+(t)}\}\Big\}. $$
\end{enumerate}
\end{remark}

Now, by Proposition 10 in \cite{Ho1} and in view of \eqref{admissible}
we obtain:
\begin{equation}
\label{locadm}
\kappa(t) \in \left[-\frac{1}{s^-(t)}, \frac{1}{s^+(t)}\right]\mbox{ for almost every }t\in (0, T).
\end{equation}
Moreover, the Lipschitz map $\Phi$ with:
\begin{equation}
\label{det}
\det\nabla\Phi(s,t) = -(1 - s\kappa(t)) \quad \mbox{ for almost every }(s,t)\in M_{s^{\pm}},
\end{equation}
is a homeomorphism from $M_{s^{\pm}}$ onto $\Omega$.
We will frequently make the extra assumption that mean curvature of $u$ be bounded
away from zero:
\begin{equation}
\label{convex}
|\mbox{trace } [a_{ij}](x)| > 0\mbox{ for all }x\in\overline{\Omega}.
\end{equation}
Then \eqref{A} and \eqref{LL^11} imply
that there is $\delta > 0$ such that
\begin{equation}
\label{uniadm}
\kappa(t) \in \left[\delta-\frac{1}{s^-(t)}, \frac{1}{s^+(t)} - \delta\right]
\mbox{ for almost every }t\in (0, T).
\end{equation}
By Proposition 10 (iii) in \cite{Ho1} and the bounds \eqref{uniadm},
$\Phi^{-1}$ is Lipschitz as well (this assertion is
generally false if \eqref{uniadm} is violated).

\begin{lemma}
\label{lemregular}
Assume that $S$ is developable of class $\mathcal{C}^{k, 1}$ for some $k\geq 2$ and
that \eqref{convex} holds. Then
$\kappa, \kappa_n\in \mathcal{C}^{k-2, 1}$ and $\Phi$, $\Phi^{-1}$ are $\mathcal{C}^{k-1, 1}$
up to the boundary of their respective domains $M_{s^{\pm}}$ and $\Omega$.
\end{lemma}
\begin{proof}
The hypothesis on $S$ imply that $a_{ij}\in \mathcal{C}^{k-2,1}(\overline{\Omega})$.
By continuity of $a_{ij}$ and $\Phi$ and by \eqref{convex}, we may assume without loss of
generality that $\Tr [a_{ij}](\Phi) \geq c > 0$ on $M_{s^{\pm}}$. In
view of \eqref{A} we have:
\begin{equation}
\label{lemreg-1}
\Tr [a_{ij}](\Phi) = \frac{\kappa_n}{1 - s\kappa}.
\end{equation}
Since $\Phi$ is bilipschitz, this implies that the right-hand side of \eqref{lemreg-1}
is Lipschitz. As $\kappa_n \geq c > 0$ it follows that
$\kappa, \kappa_n$ are Lipschitz. Thus $\Gamma'$, $N$ belong to $\mathcal{C}^{1,1}$,
hence so does $\Phi$. By \eqref{det}, \eqref{uniadm}, the Jacobian of $\Phi$ is uniformly
bounded away from zero on $M_{s^{\pm}}$, and so $\Phi^{-1}$ belongs to $\mathcal{C}^{1,1}$, too.
If $k \geq 3$ then we return to \eqref{lemreg-1},  apply Lemma \ref{chainrule} and  
argue as before to conclude that
$\kappa, \kappa_n$ are in $\mathcal{C}^{1,1}$. The conclusion follows by iteration.
\end{proof}

In the proof of Lemma \ref{lemregular} we used the following
chain rule, which is a particular case of Theorem 2.2.2 from \cite{ziemer}.

\begin{lemma}\label{chainrule}
Let $U_1, U_2\subset \RR^n$ be two open, bounded sets and let $\Phi : U_1\to U_2$
be a bilipschitz homeomorphism.
Then $f\in W^{1,2} (U_2)$ if and only if $f\circ \Phi \in W^{1,2} (U_1)$. If this is the
case, then the chain rule applies:
\begin{equation}\label{cr}
\nabla (f\circ \Phi) = \big((\nabla f) \circ \Phi\big) \nabla \Phi \quad \mbox{ a.e. in }\,\, U_1.
\end{equation}
\end{lemma}

\section{Equation $\mbox{sym}\nabla w = B$ on developable surfaces}\label{equation}

In this section we let $S$ be a developable surface of class $\mathcal{C}^{2,1}$.
By $\vec n : S\to\mathbb{R}^3$ we denote the unit normal to $S$ which satisfies
$|\vec n(x)|=1$ and $\vec n(u(x)) = \d_1 u(x)\times \d_2 u(x)$ for all $x\in\Omega$, and
$\vec n(\gamma(t)) = n(t)$ for all $t\in (0, T)$. By $\Pi = \nabla \vec n$ we denote
the second fundamental form of $S$, defined as a symmetric bilinear form by:
$$ \displaystyle \Pi (p) (\tau, \eta) = \eta\cdot \partial_\tau \vec n, \quad \forall \,\, 
\tau,\eta\in T_pS, \,\, p\in S, $$
where $T_pS$ is the tangent plane to $S$ at point $p$, so that:
$$ \Pi(u(x))(\d_i u(x), \d_j u(x)) =- a_{ij}(x)\mbox{ for all }x\in
\Omega. $$

We continue to assume \eqref{convex}. Hence:
\begin{equation}\label{bdd}
|\kappa_n(t)|>c>0 \qquad\forall t \in [0,T].
\end{equation}
For a given symmetric bilinear form
$B\in\mathcal{C}^{1,1}(S),\RR^{2\times 2}$, we want to solve a first order PDE:
\begin{equation}\label{prob}
\mbox{sym} \nabla w = B,
\end{equation}
on $S$, where $w:S\to \mathbb R^3$ is a displacement field and the expression sym$\nabla w$ in
the left-hand side is the following bilinear form acting on the tangent space of 
$S$:
\begin{equation*}
  \displaystyle \mbox{sym}\nabla w(p) (\tau, \eta)= \frac 12
  (\partial_\tau w(p) \cdot \eta + \partial_\eta w(p) \cdot \tau), \quad \forall \,\, 
\tau,\eta\in T_pS, \,\, p\in S.
\end{equation*}

\smallskip

{\bf 1.} We shall write $w\vec n$ to denote the scalar product
$w\cdot\vec n$, and we decompose $w$ as follows:
$$ w = w_{tan} + (w\vec n)\vec n. $$
Hence $w_{tan}(p)\in T_p S$ for all $p\in S$. 
We define the pulled back maps $w_3 = (w\vec n)\circ u$ and 
$w' = (w_{tan}\circ u)^T \nabla u$, 
as well as the pulled back form:
$$ B_{ij}(x) = B(u(x))(\partial_i u(x), \partial_j u(x)) \qquad
\forall x\in \Omega \quad \forall i,j = 1,2. $$
Using \eqref{uijAij} and recalling that $w_{tan}$ is tangent to $S$, we calculate:
\begin{equation*}
\begin{split}
\partial_j u(x)\cdot \nabla w (u(x))\partial_i u(x)
& = \partial_i (w(u(x)))\cdot \partial_j u(x) 
= \partial_i \Big(w(u(x)) \cdot \partial_j u(x)\Big) -
w(u(x))\cdot \partial^2_{ij}u(x) \\ &
= \partial_i w'_j - w_3 a_{ij} (x).
\end{split}
\end{equation*}
Hence \eqref{prob} can be
written in terms of the pulled back quantities as
the following matrix equality:
\begin{equation}\label{prob1}
\begin{split}
[B_{ij}] = & \mbox{ sym} \nabla w' - w_3[a_{ij}],
\end{split}
\end{equation}
where now sym $\nabla w'$ is understood in the usual way, with respect to the standard 
Euclidean coordinates in $\Omega$.

\smallskip

{\bf 2.}   Recalling that the condition for a matrix field
$\tilde B$ to be of the form $\tilde B = \mathrm{sym} \nabla \tilde w$
for some vector field $\tilde{w}$
on $\Omega$ is equivalent to $\mbox{curl}^T\mbox{curl} \tilde B = 0$
(in view of $\Omega$ being simply connected), the equation (\ref{prob1})
becomes:
\begin{equation}\label{prob2}
\begin{split}
\mbox{curl}^T&\mbox{curl } [B_{ij}] = - \mbox{curl}^T\mbox{curl}
\Big(w_3 [a_{ij}]\Big) \\ & = -w_3 \mbox{ curl}^T\mbox{curl } [a_{ij}]
- 2 \nabla^\perp w_3\cdot \mbox{curl } [a_{ij}]
- \mbox{cof } \nabla^2w_3 : [a_{ij}].
\end{split}
\end{equation}
Notice now that:
\begin{equation*}
\begin{split}
\mbox{curl}~ [a_{ij}] = -
\mbox{curl}\left[\begin{array}{cc}\partial_{11}^2u\cdot n
    & \partial_{12}^2u\cdot n \\
\partial_{12}^2u\cdot n & \partial_{22}^2u\cdot n\end{array}\right] =
- \left[\begin{array}{c}\partial_{11}^2 u\cdot \partial_2 n
   - \partial_{12}^2u\cdot  \partial_1 n \\
\partial_{12}^2u\cdot  \partial_2 n
- \partial_{22}^2u\cdot  \partial_1 n\end{array}\right]  =0,
\end{split}
\end{equation*}
because $ \partial_i(\vec n\circ u)\in T_pS$ and
$ \partial^2_{ij}u\cdot\partial_k u=0$ by Lemma \ref{lem1}.
Hence:
\begin{equation}
\label{defalpha}
\theta = \mbox{curl}^T\mbox{curl } [B_{ij}]
\end{equation}
belongs to $L^\infty$. The problem (\ref{prob2}) becomes:
\begin{equation}\label{prob3}
\theta = -\mbox{cof } \nabla^2w_3 : [a_{ij}]
\qquad \mbox{ in } \Omega.
\end{equation}

\smallskip

{\bf 3.} By Lemma \ref{lem1} we have:
\begin{equation*}
\left(\mbox{cof } (\nabla^2 w_3) : [a_{ij}]\right)(\Phi(s,t)) =
\frac{\kappa_n(t)}{1-s\kappa(t)} \partial^2_{ss} (w_3(\Phi(s, t))).
\end{equation*}
Consequently, problem (\ref{prob3}) is equivalent to:
\begin{equation}\label{prob4}
\partial^2_{ss}(w_3(\Phi(s,t))) = - \frac{1-s\kappa(t)}{\kappa_n(t)}
\theta(\Phi(s,t)) \qquad \mbox{ for all } (s,t)\in M_{s^{\pm}}.
\end{equation}

The above calculations show that in order to solve (\ref{prob}), it
is sufficient and necessary to solve  the ODE (\ref{prob4}) for $w_3$
and then recover $w'$ from (\ref{prob1}). Moreover, the solution $(w', w_3)$
is unique after choosing the boundary conditions $w_3(\Phi(t,0))$
and $\partial_s (w_3(\Phi(t,0))$, where uniqueness of $w'$ is understood
up to affine (linearized) rotations of the
form $A(s,t) +b$, $A\in so(2)$, $b\in \mathbb{R}^2$.

\medskip

\begin{theorem}\label{prop1}
Assume that $S$ is developable of class $\mathcal{C}^{2,1}$ and
satisfies \eqref{convex}, and let $\alpha\in (0, 1)$. Then there
exists a constant $C$ such that the following is true.
For every symmetric bilinear form
$B\in\mathcal{C}^{1,1}(S,\RR^{2\times 2})$ 
there exists a solution $w = w_{tan} + (w\vec n)\vec n$ with $w_{tan}\in \mathcal{C}^{0,\alpha}$ and
$(w\vec n)\in L^\infty$ of:
\begin{equation}
\label{bd0}
\mathrm{sym} \nabla w = \mathrm{sym}\nabla w_{tan} + (w\vec n) \Pi = B
\end{equation}
satisfying the bounds:
\begin{equation}\label{bd1}
\|w_{tan}\|_{\mathcal{C}^{0,\alpha}} + \|(w\vec n)\|_\infty \le
C\|B\|_{\mathcal{C}^{1,1}}.
\end{equation}
If, in addition, $S\in\mathcal{C}^{k+2,1}$ and $B\in \mathcal{C}^{k+1, 1}$ for some $k\geq 1$, then
\begin{equation}\label{bd2}
\|w_{tan}\|_{\mathcal{C}^{k,1}} + \|(w\vec n)\|_{\mathcal{C}^{k-1,1}} \le C \|B\|_{\mathcal{C}^{k+1,1}}.
\end{equation}
\end{theorem}
\begin{proof}
{\bf 1.} Assume first the minimal regularity $u\in\mathcal{\mathcal{C}}^{2,1}$ and
$B\in\mathcal{C}^{1,1}$, so that $\theta=\mbox{curl}^T\mbox{curl}
[B_{ij}]\in L^\infty(\Omega,\mathbb{R})$. Solving (\ref{prob4}) by integrating twice
in $s$ from  $w_3(\Phi(t,0)) = 0$ and $\partial_s w_3(\Phi(t,0))=0$, we obtain:
\begin{equation}\label{a}
\|(w\vec n)\|_{L^\infty} \le C \|\theta \|_{L^\infty}  \le C \|B\|_{\mathcal{C}^{1,1}}.
\end{equation}
Solving now (\ref{prob1}) for $w_{tan}$ so that:
$$\mbox{skew}\fint_\Omega\nabla (w_{tan}\circ u) = 0 \quad \mbox{ and} \quad
\fint_\Omega (w_{tan}\circ u) = 0,$$
(where for a square matrix $P$, its skew-symmetric part is denoted
$\mbox{skew } P = \frac{1}{2} (P + P^T)$), 
we obtain by means of Korn's inequality, for any $p > 1$:
\begin{equation*}
\|\nabla w_{tan}\|_{L^p}\leq C\|\mbox{sym} \nabla (w_{tan}\circ
u)\|_{L^p}  \le C (\|[B_{ij}]\|_{L^\infty} + \|(w\vec n)\|_{L^\infty}) \le C \|B\|_{\mathcal{C}^{1,1}},
\end{equation*}
where $C$ may depend on $p$. Combining with the Poincar\'e
inequality, we get:
\begin{equation}\label{b}
\|w_{tan}\|_{W^{1,p}} \le C \|B\|_{\mathcal{C}^{1,1}}.
\end{equation}
By Sobolev embedding, (\ref{bd1}) follows now from (\ref{a}) and
(\ref{b}).

\smallskip

{\bf 2.} When $B\in\mathcal{C}^{k+1,1}$ and
$u\in\mathcal{C}^{k+2,1}$, then
$-\frac{1-s\kappa}{\kappa_n}\theta\in \mathcal{C}^{k-1,1}$ by Lemma
\ref{lemregular}, and so by (\ref{prob4}):
\begin{equation}\label{c}
\|(w\vec n)\|_{\mathcal{C}^{k-1,1}} \le C
\|\theta\|_{\mathcal{C}^{k-1,1}} \leq C \|B\|_{\mathcal{C}^{k+1,1}},
\end{equation}
where $C$ may depend on $S$. Recalling that $\nabla^2 w'$ can be
expressed as the linear combination of partial derivatives of
$\mbox{sym}\nabla w'$, we have from \eqref{bd0}, \eqref{c} that:
$$ \|\nabla^2w_{tan}\|_{\mathcal{C}^{k-2,1}} \leq C \|B\|_{\mathcal{C}^{k+1,1}},$$
which implies (\ref{bd2}) in view of (\ref{bd1}) and (\ref{c}).
\end{proof}

\begin{proposition}\label{prop2}
Assume that $S$ is developable of class $\mathcal{C}^{k+2,1}$,
satisfying (\ref{convex}). Also assume that $B = \mathrm{sym}
((\nabla \phi)^T (\nabla \psi))$ where $\phi, \psi \in
\mathcal{C}^{k+1,1}(S, {\mathbb R}^3)$. Then $w$ as obtained in
Theorem \ref{prop1} satisfies:
$$ \|w_{tan}\|_{\mathcal{C}^{k,1}} + \|w\vec n\|_{\mathcal{C}^{k-1,1}}
\le C \|\psi\|_{\mathcal{C}^{k+1,1}} \|\phi\|_{\mathcal{C}^{k+1,1}}.$$
\end{proposition}
\begin{proof}
After straightforward calculations, we obtain:
$$ B_{ij} = \frac 12 \big(\partial_i (\phi\circ u) \cdot \partial_j
(\psi\circ u) + \partial_j(\phi\circ u) \cdot \partial_i (\psi\circ u)\big). $$
Further calculations shows that in the expansion of $\theta$, as defined in \eqref{defalpha},
the third derivatives of $\psi$ and $\phi$ cancel out:
\begin{equation*}
\begin{split}
\theta &= \mbox{curl}^T\mbox{curl } [B_{ij}]
= \partial^2_{11} B_{22} + \partial^2_{22} B_{11} - 2\partial^2_{12} B_{12} \\
& =  \partial^2_{22} \big(\partial_1 (\phi\circ u) \cdot \partial_1 (\psi\circ u)\big)
+ \partial^2 _{11} \big(\partial_2 (\phi\circ u) \cdot \partial_2(
\psi\circ u)\big)  \\
&\qquad \qquad -  \partial^2_{12} \big(\partial_1(\phi\circ u) \cdot \partial_2(\psi\circ u)
+ \partial_2 (\phi\circ u)\cdot \partial_1 (\psi\circ u)\big)\\
& = - \big(\partial^2_{11}(\phi\circ u)\cdot\partial^2_{22}(\psi\circ u)
+ \partial^2_{22}(\phi\circ u) \cdot\partial^2_{11}(\psi\circ u)
- 2 \partial^2_{12}(\phi\circ u)\cdot\partial^2_{12}(\psi\circ u)\big).
\end{split}
\end{equation*}
As a consequence, if $S$ is of class $\mathcal{C}^{k+2,1}$, the solution of (\ref{prob4}) from Theorem
\ref{prop1} satisfies:
$$ \|w\vec n\|_{\mathcal{C}^{k-1,1}} \le C \|\theta\|_{\mathcal{C}^{k-1,1}} \le C
\|\phi\|_{\mathcal{C}^{k+1,1}} \|\psi\|_{\mathcal{C}^{k+1,1}}. $$
Reasoning as in the proof of Theorem \ref{prop1}, we obtain:
$$ \|w_{tan}\|_{\mathcal{C}^{k,1}} + \|w\vec n\|_{\mathcal{C}^{k-1,1}} \le C \|B\|_{\mathcal{C}^{k-1,1}}
+  C \|\phi\|_{\mathcal{C}^{k+1,1}} \|\psi\|_{\mathcal{C}^{k+1,1}}
\le C  \|\phi\|_{\mathcal{C}^{k+1,1}} \|\psi\|_{\mathcal{C}^{k+1,1}},$$
proving the claim.
\end{proof}

\section{Spaces of $W^{2,2}$ infinitesimal isometries on developable surfaces}\label{space}

In this section, we establish some properties of $W^{2,2}$ first order infinitesimal
isometries on developable surfaces $S$ of $\mathcal{C}^{2,1}$ regularity.
We give a classification of these displacements and prove
that they are necessarily $\mathcal{C}^{1,1/2}$ regular.
Define:
\begin{equation*}
\mathcal{V} = \big\{V\in W^{2,2} (S, {\mathbb R}^3); ~\mbox{sym} \nabla V =0\big\}
\end{equation*}
Note that in view of Lemma \ref{chainrule}, we may freely determine the regularity of any mapping 
on $S$, up to $\mathcal{C}^{2,1}$ regularity, by considering the regularity of its composition with the chart $u$. 
Here and in what follows we write $f\in W^{2,2}(S)$ precisely if
$f\circ u\in W^{2,2}(\Omega)$.  

The following is the main result of this section:

\begin{theorem}\label{prop1-1storder}
Let $V\in\mathcal{V}$ and assume that $S$ is developable of class $\mathcal{C}^{2,1}$.
Then $V\in \mathcal{C}^{1,1/2}(S,\mathbb{R}^3)$.  More precisely,
writing  $V= V_{tan} + (V\vec n) \vec n$ we have:
$$ V_{tan} \in \mathcal{C}^{2,1/2}(S, \RR^3) \quad
\mbox{ and } \quad V\vec n\in \mathcal{C}^{1,1/2} (S). $$
Moreover, $V\in \mathcal V$ if and only if there exist $a, b\in W^{2,2}((0, T),\mathbb{R})$ such that
\begin{align}
\label{V3} (V\vec n) (u(\Phi(s, t)))  &= a(t) + sb(t),
\\
\label{e} \mathrm{sym} \nabla V_{tan} (u(\Phi(s, t))) &= \frac{a(t) + sb(t)}{1 - s\kappa(t)}\kappa_n(t)\
(\Gamma'(t)\otimes\Gamma'(t)) \qquad \mbox{ for a.e. } (s, t)\in M_{s^{\pm}},
\end{align}
and such that the following integrals are finite:
\begin{align}
\label{est1}
J_1(a,b) &= \int_{M_{s^{\pm}}}\left( b'(t) + \frac{\kappa(a'(t) + sb'(t))}{1 - s\kappa(t)}
\right)^2 \frac{\mathrm{d}s\mathrm{d}t}{1 - s\kappa(t)} < \infty,
\\
\label{est2}
J_2(a,b) &= \int_{M_{s^{\pm}}} \left(
a''(t) + sb''(t) - \kappa(t) (1-s\kappa(t)) b(t) + \frac{s\kappa'(t)(a'(t) + sb'(t))}{1 - s\kappa(t)}
\right)^2  \frac{\mathrm{d}s\mathrm{d}t}{(1 - s\kappa(t))^3} <\infty.
\end{align}
\end{theorem}

\begin{proof}
Since $u\in \mathcal{C}^{2,1}$, from \eqref{A} we conclude that $\kappa_n$ is continuous
up to the boundary and that $\kappa$ is continuous on the open set where $\kappa_n$
differs from zero.

{\bf 1.} Let $V\in\mathcal{V}$.
From the observations following Lemma \ref{lem1} and from \eqref{det} we deduce that
$\Phi$ is bilipschitz on every set of the form:
$$ M'_{\delta} = \{(s, t)\in M_{s^{\pm}} : s\in (\delta - s^-(t),
s^+(t) - \delta)\} $$
with $\delta > 0$. Indeed, $\Phi^{-1}$ may fail to be globally Lipschitz on
$\Phi(M_{s^{\pm}})$ unless \eqref{uniadm} is satisfied.
Set $V_3 = (V\vec n)\circ u$ and
$f = V_3\circ \Phi$. We have $V_3\in W^{2,2}(\Omega)$.
By a similar reasoning as in Lemma \ref{lemregular}, we obtain that
$\Phi$ is a $\mathcal{C}^{1,1}$ diffeomorphism on $M'_\delta$.
Hence, Lemma \ref{chainrule} implies that $f\in W^{2,2}(M'_{\delta})$ with:
\begin{equation} \label{grad1}
\partial_s f(s,t) = \nabla V_3(x) N(t), \quad
\partial_t f(s,t) = (1-s\kappa(t))\nabla V_3(x)\Gamma'(t) \quad
\mbox{ where } x=\Phi(t,s),
\end{equation}
\begin{equation}\label{grad2}
\begin{split}
\partial^2_{ss}  f(s,t) &= \big(\nabla^2 V_3(x) N(t)\big) N(t),\\
\partial^2_{ts} f(s,t) &= (1-s\kappa(t))\big(\nabla^2 V_3(x)\vec
N(t)\big) \Gamma'(t)  - \kappa(t) \nabla V_3(x)\Gamma'(t),\\
\partial^2_{tt} f(s,t) &= (1-s\kappa(t))^2\big(\nabla^2
V_3(x)\Gamma'(t)\big) \Gamma'(t) +
\nabla V_3(x) \Big( \kappa (1-s\kappa) N(t)  - s\kappa'\Gamma'(t)\Big).
\end{split}
\end{equation}
Moreover, by (\ref{prob4}):
\begin{equation}\label{ss}
\partial^2_{ss} f(s,t) = 0 \qquad \forall t\in[0,T]\mbox{ with }\kappa_n(t)\neq 0.
\end{equation}
Indeed, $\theta = \mbox{curl}^T\mbox{curl }[B_{ij}]  =0$ in the
present case where $\mbox{sym}\nabla V=0$.

Let $0<\eta < \inf_{t\in [0,T]} \{s^-(t), s^+(t)\}$ so that
$ (-\eta, \eta)\times (0,T)\subset M_{s^{\pm}}$.
Denote by $f^*$ the precise representative of
$f$ \cite{EG} and define:
$$a(t) = f^*(0,t),\qquad b(t) = \frac{1}{\eta} (f^*(\eta,t) - a(t)).$$
By (\ref{prob4}) the definition of $b$ does not depend on $\eta$, and:
\begin{equation}\label{pr1cl1}
f(s,t)=a(t) + sb(t) \qquad\forall (s,t)\in M_{s^{\pm}}.
\end{equation}

\medskip

{\bf 2.} Since $f\in W^{2,2}((0,T)\times (-\eta, \eta))$,
for almost every pair $s_1, s_2\in (-\eta, \eta)$ the traces $f(\cdot, s_1)$ and $f(\cdot, s_2)$
belong to $W^{2,2}(0,T)$, by Fubini's theorem. Hence $b=
\frac{1}{s_1-s_2}(f(\cdot, s_1) - f(\cdot, s_2)) \in W^{2,2}((0,T))$,
and $a\in  W^{2,2}((0,T))$ as well.
By Sobolev embedding $f\in \mathcal{C}^{1,1/2}(M_{s^{\pm}})$.

Since $\Phi$ is a $\mathcal{C}^{1,1}$ diffeomorphism, it also follows
that $V\vec  n$, $V_{tan} \in\mathcal{C}^{1,1/2}(S)$ which implies \cite{lemopa1}
that $V_{tan}\in\mathcal{C}^{2,1/2}$. Finally, (\ref{e}) follows from
Lemma \ref{lem1}.

\medskip

{\bf 3.}  We shall now prove that, given the structure (\ref{pr1cl1}),
condition $V\vec n\in W^{2,2}(S)$ is equivalent to $a,b\in W^{2,2}(0,T)$
satisfying (\ref{est1}), (\ref{est2}). This will conclude the proof.

Inserting (\ref{grad1}) into (\ref{grad2}) we obtain, for all $t\in
[0,T]\setminus I_0$:
\begin{equation*}
\begin{split}
b'(t) &= \partial^2_{ts}f(s,t) = (1 - s\kappa)\big(\nabla^2 V_3(x)\vec
N(t)\Big)\Gamma'(t)  - \frac{\kappa}{1 - s\kappa} (a'(t) + sb'(t)),\\
a''(t) + sb''(t) &= \partial^2_{tt}f(s,t)= (1 -
s\kappa)^2\big(\nabla^2 V_3(x) \Gamma'(t)\big) \Gamma'(t)
+ \kappa(1-s\kappa) b(t) - \frac{s\kappa'}{1 - s\kappa}(a'(t) + sb'(t)).
\end{split}
\end{equation*}
Solving for  $\nabla^2 V_3(x)$ we get:
\begin{equation*}
\begin{split}
\big(\nabla^2V_3(x) \Gamma'(t)\big) \Gamma'(t) & = \frac{1}{(1 -
  s\kappa)^2} \Big(a''(t) + sb''(t) - \kappa (1-s\kappa)b(t) + \frac{s\kappa'}{1-s\kappa}
(a'(t) + sb'(t))\Big),\\
\big(\nabla^2V_3(x) N'(t)\big) \Gamma'(t)& = \frac{1}{1 -s\kappa}
\Big(b'(t) + \frac{\kappa'}{1 - s\kappa} (a'(t) + sb'(t)\Big),\\
\big(\nabla^2V_3(x) N'(t)\big) N'(t)& = 0.
\end{split}
\end{equation*}
Now a change of variables shows that:
\begin{equation*}
\int_{\Omega} |\nabla^2 V_3(x)|^2\mbox{d}x = \int_{M_{s^{\pm}}} |(\nabla^2 V_3)(\Phi(s, t))|^2
(1 - s\kappa)\mbox{d}s\mbox{d}t.
\end{equation*}
We see that  $V_3\in W^{2,2}(\Omega)$ if and only if  (\ref{est1}) and
(\ref{est2}) hold.
\end{proof}

We finish this section by pointing out a straightforward corollary of
the above calculations:
\begin{proposition}\label{pr2}
Let $v\in W^{2,2}(S)$ satisfy
\begin{equation*}\label{V3'}
v(u(\Phi(s,t))) = a(t) + sb(t) \quad \mbox{ for a.e. }(s, t) \in M_{s^{\pm}}.
\end{equation*}
Then $a, b\in W^{2,2}(0,T)$ and there exists a tangent vector field
$V_{tan}\in W^{2,2}(S,\mathbb{R}^3)$ to $S$ such that
$V_{tan} + v \vec n \in{ \mathcal V}$.
\end{proposition}

\section{Matching and density of infinitesimal isometries}

\begin{definition}
A one parameter family $\{u_\e\}_{\e > 0}\subset \mathcal{C}^{0,1}(\overline{S}, {\mathbb R}^3)$ is said
to be a (generalized) $N$th order infinitesimal isometry if the change
of metric induced by $u_\e$ is of order $\e^{N+1}$, that is:
\begin{equation}\label{iso}
\|(\nabla u_{\e})^T(\nabla u_{\e}) - \mathrm{Id}\|_{L^{\infty}(S)} = \mathcal{O}(\e^{N+1})
\mbox{ as }\e\to 0.
\end{equation}
\end{definition}
Here and in what follows we use the Landau symbols $\mathcal{O}(q)$ and $o(q)$.
They denote, respectively, any quantity whose quotient with $q$ is uniformly
bounded or converges to $0$ as $q\to 0$.
%

Note that if $V\in \mathcal V \cap C^{0,1}$, then
$u_\e= \mbox{id} + \e V$ is a (generalized) first order isometry.

\begin{theorem}\label{th_nth-intro}
Let $S$ be a developable surface of class $\mathcal{C}^{2N,1}$, satisfying (\ref{convex}).
Given $V\in\mathcal{V}\cap\mathcal{C}^{2N-1,1}(\bar S)$, there exists
a family $\{w_\varepsilon\}_{\e > 0}\subset \mathcal{C}^{1,1}(S, \mathbb{R}^3)$, equibounded in
$\mathcal{C}^{1,1}(S)$, such that for all small $\varepsilon > 0$ the family:
$$ u_{\varepsilon} = \mathrm{id} +\varepsilon V + \varepsilon^2
w_\varepsilon $$
is a (generalized)
$N$th order isometry of class $\mathcal{C}^{1,1}$.
\end{theorem}
\begin{proof}
{\bf 1.} The result is a consequence of the following claim.
Let $S$ be of class $\mathcal{C}^{k+2,1}$ and let $u_\e$ be an $(i-1)$th order isometry of regularity
$\mathcal{C}^{k+1,1}$ of the form:
$$  u_\e= \mbox{id} + \sum_{j=1}^{i-1}\e^j w_j, \qquad w_j\in \mathcal{C}^{k+1,1}.$$
Then there exists $w_i\in\mathcal{C}^{k-1,1}(S,\mathbb{R}^3)$ so that
$\phi_\e= u_\e + \e^i w_i$ is an $i$th order infinitesimal isometry,
and:
\begin{equation}\label{new}
\|w_i\|_{\mathcal{C}^{k-1,1}} \le C \sum_{j=1}^{i-1} \|w_j\|_{\mathcal{C}^{k+1,1}}
\|w_{i-j}\|_{\mathcal{C}^{k+1,1}}.
\end{equation}
Indeed, setting $w_1=V\in\mathcal{C}^{2N-1,1}$ and applying the above
result iteratively to find $w_j\in\mathcal{C}^{2N-2j+1,1}$, for
$j = 2\ldots N$,  we obtain the requested $w_\e = w_2 + \e w_3 + \cdots \e ^{N-2}
w_N\in\mathcal{C}^{1,1}$.

\smallskip

{\bf 2.} We now prove the claim. Set $w_0= \mbox{id}$. Calculating the change of metric
induced by the deformation $\phi_\e$ we get:
$$ |(\nabla \phi_\e)^T\nabla \phi_\e - \mbox{Id}| =
\left|\sum_{j=1}^i \e^j A_j\right| + \mathcal{O}(\e^{i+1}), $$
where the expression $A_j$ indicating the change of metric of  $j$th
order induced by $\phi_\e$, is given by:
$$ A_j =  \sum_{p=0}^j \mbox{sym} \Big( (\nabla w_p)^T \nabla w_{j-p}
\Big ). $$
Note that by the assumption $A_1=\cdots= A_{i-1}=0$.
Consequently, in order for $\phi_\e$ to be an $i$th order isometry, we must have
$A_i=0$ or equivalently:
$$ \mbox{sym} \nabla w_i = -\frac12 \sum_{p=1}^{i-1} \mbox{sym} \Big(
(\nabla w_p)^T \nabla w_{i-p} \Big ).$$
Applying Theorems \ref{prop1} and \ref{prop2}, we obtain that such $w_k$ exists with the estimate:
$$ \|w_{i,tan}\|_{\mathcal{C}^{k,1}} + \|w_{i,3}\|_{\mathcal{C}^{k-1,1}} \le C
\sum_{p=1}^{i-1} \|w_p\|_{\mathcal{C}^{k+1,1}}\|w_{i-p}\|_{\mathcal{C}^{k+1,1}}, $$
provided that all $w_p\in\mathcal{C}^{k+1,1}$ and that $S$ is of class $\mathcal{C}^{k+2, 1}$.
This completes the proof of the claim and of the theorem.
\end{proof}

\begin{theorem}\label{th_density-intro}
Assume that $S$ is developable, of class $\mathcal{C}^{k+1,1}$ up
to the boundary, and satisfying (\ref{convex}).
Then, for every $V\in\mathcal{V}$ there exists a sequence
$V_n\in\mathcal{V}\cap\mathcal{C}^{k,1}(\bar S,\mathbb{R}^3) $
such that: $$\lim_{n\to\infty} \|V_n - V\|_{W^{2,2}(S)} = 0.$$
\end{theorem}

\begin{proof}
Let $a,b\in W^{2,2}(0,T)$ be as in Proposition \ref{prop1-1storder}.
Take $a_n, b_n \in \mathcal{C}^\infty ([0,T])$ converging in $W^{2,2}$
to $a$, $b$ respectively, and define:
$$ v_n(s,t)= a_n(t) + s b_n(t). $$  
By Propositions \ref{prop1-1storder} and \ref{pr2}, there
exist $V_n \in \mathcal V$ such that $(V_n \vec n ) \circ u \circ \Phi = v_n$ and
$\|V_n\vec n - V \vec n\|_{W^{2,2}(S)}\to 0$. Indeed, the last assertion is equivalent to proving that  
$J_i(a-a_n, b-b_n) \to 0$, $i=1,2$, which is established immediately after observing that $1-s\kappa$ is 
bounded away from $0$ by (\ref{uniadm}). Note that $(V_n)_{tan}$ can
be chosen suitably such that also $\|(V_n)_{tan} - V_{tan}\|_{W^{2,2}(S)}\to 0$.
In view of Lemma \ref{lemregular},
$V_n \circ u\in \mathcal{C}^{k,1} (\overline{\Omega})$, that is 
$V_n \in \mathcal{C}^{k,1}$ up to the boundary of $S$.
\end{proof}

\section{The $\Gamma$-limit result}

Consider a family $\{S^h\}_{h>0}$ of thin shells of thickness $h$ around $S$:
$$S^h = \{z=p + t\vec n(p); ~ p\in S, ~ -h/2< t < h/2\}, \qquad 0<h<h_0,$$
where $h_0$ is small enough so that the projection map $\pi : S^{h_0}\to S, \pi(p+ t \vec n(p) := p$ is well-defined.  
For a $W^{1,2}$ deformation $u^h: S^h\rightarrow \mathbb{R}^3$,
we assume that its elastic energy (scaled per unit thickness)
is given by the nonlinear functional:
$$E^h(u^h) = \frac{1}{h}\int_{S^h} W(\nabla u^h).$$The stored-energy density function $W:\mathbb{R}^{3\times 3}\longrightarrow
[0,\infty]$ is $\mathcal{C}^2$  in an open neighborhood of $SO(3)$,
and it is assumed to satisfy the conditions of normalization, frame
indifference and quadratic growth:
\begin{equation*}
\begin{split}
\forall F\in \mathbb{R}^{3\times 3} \quad
\forall R\in SO(3) \qquad
&W(R) = 0, \quad W(RF) = W(F), \\
&W(F)\geq C \mathrm{dist}^2(F, SO(3)),
\end{split}
\end{equation*}
with a uniform constant $C>0$.
The potential $W$ induces the quadratic forms:
\begin{equation*}\label{Qmatrices}
\mathcal{Q}_3(F) = D^2W(\mbox{Id})(F,F), \qquad
\mathcal{Q}_2(p, F_{tan}) = \min\{\mathcal{Q}_3(\tilde F); ~~ (\tilde F - F)_{tan} = 0\}.
\end{equation*}
defined for  $F\in\mathbb{R}^{3\times 3}$, and
$p\in S$ respectively. Here and in what follows $F_{tan}$ is the bilinear form induced by $F$ on $S$
through the formula:
\begin{equation*}
\displaystyle F_{tan} (\tau, \eta) = \tau \cdot F(p) \eta \quad \forall p\in S,\,  \tau, \eta \in T_pS.
\end{equation*}
Both forms $\mathcal{Q}_3$ and all $\mathcal{Q}_2(p,\cdot)$ depend only on the symmetric parts of their arguments,
with respect to which they are positive definite \cite{FJMgeo}.

In what follows we shall consider a sequence $e^h>0$ such that:
\begin{equation}\label{scaling-intro}
0<\lim_{h\to 0} e^h/h^\beta < +\infty, \qquad \mbox{ for some }~~ 2<\beta<4.
\end{equation}
Also, let:
$$\beta_N = 2+ 2/N.$$ 
We assume that $N>1$ (the case $N=1$ is already
covered in \cite{lemopa1}).
Recall the following result:
\begin{theorem}\label{thliminf-intro}\cite{lemopa1}.
Let $S$ be a surface embedded in $\mathbb{R}^3$, which is
compact, connected, oriented, of class $\mathcal{C}^{1,1}$,
and whose boundary $\partial S$ is the union of finitely many
Lipschitz curves.
Let $u^h\in W^{1,2}(S^h,\mathbb{R}^3)$ be a sequence of deformations whose scaled energies
$E^h(u^h)/e^h$ are uniformly bounded.
Then there exist a sequence $Q^h\in SO(3)$ and $c^h\in\mathbb{R}^3$
such that for the normalized rescaled deformations:
$$y^h(p+t\vec{n}) = Q^h u^h(p+h/h_0 t\vec{n}) - c^h$$
defined on the common domain $S^{h_0}$, the following holds.
\begin{enumerate}
\item[(i)] $y^h$ converge in $W^{1,2}(S^{h_0})$ to $\pi$.
\item[(ii)]  The scaled average displacements:
\begin{equation}\label{Vh-intro}
V^h(p) = \frac{h}{\sqrt{e^h}} \fint_{-h_0/2}^{h_0/2}
y^h(p+t\vec{n}) - p ~\mathrm{d}t
\end{equation}
converge (up to a subsequence) in $W^{1,2}(S)$ to some $V\in \mathcal{V}$.
\item[(iii)] $\liminf_{h\to 0} {1}/{e^h} E^h(u^h) \geq \mathcal{I}(V)$, where:
\begin{equation}\label{limitenergy}
\mathcal{I}(V) = \frac{1}{24}
\int_S \mathcal{Q}_2\Big(p,\big(\nabla(A\vec n) - A\Pi\big)_{tan}\Big) \, \mbox{d}p.
\end{equation}
Here,  the matrix field $A\in W^{1,2}(S,\mathbb{R}^{3\times 3})$ is
such that:
\begin{equation*}\label{Vass1}
\partial_\tau V(p) = A(p)\tau \quad \mbox{and} \quad  A(p)\in so(3) \qquad
\forall {\rm{a.e.}} \,\, p\in S \quad \forall \tau\in T_p S.
\end{equation*}
\end{enumerate}
\end{theorem}

In order to prove that the linear bending functional (\ref{limitenergy}) restricted to
${\mathcal V}$ is the $\Gamma$-limit of the rescaled three dimensional nonlinear elasticity
energy $(1/e^h) E^h$ we also need to establish the limsup counterpart of the $\Gamma$-convergence
statement. This is the final contribution of this paper which we are formulating in the following theorem.
For a full discussion of this topic in this context see \cite{FJMhier}
and \cite{lemopa1}.

\begin{theorem}\label{thlimsup-intro}
Let $N>1$ and assume that $S$ is developable of class $\mathcal{C}^{2N,1}$
and satisfying (\ref{convex}). Assume that
\begin{equation}\label{ass33}
e^h= o(h^{\beta_N}).
\end{equation}
Then for every $V\in\mathcal{V}$ there exists a sequence
$u^h\in W^{1,2}(S^h,{\mathbb R}^3)$ such that:
\begin{itemize}
\item[(i)] the rescaled deformations $y^h(p+t\vec n)=u^h(p+th/h_0 \vec n)$
converge in $W^{1,2}(S^{h_0})$ to $\pi$.
\item[(ii)] the scaled average displacements $V^h$ given in (\ref{Vh-intro})
converge in $W^{1,2}(S)$ to $V$.
\item[(iii)] $\lim_{h\to 0} {1}/{e^h} E^h(u^h) = \mathcal{I}(V)$.
\end{itemize}
\end{theorem}
\begin{proof} We shall construct a recovery sequence for developable surfaces, based on
Theorems \ref{th_density-intro} and \ref{th_nth-intro}.
Indeed, by the density result and the continuity of
the functional $\mathcal{I}$ with respect to the strong topology of $W^{2,2}(S)$, we can assume
$V\in {\mathcal V}\cap{\mathcal C}^{2N-1,1}(\bar S, {\mathbb R}^3)$.
In the general case the result will then follow through  a diagonal argument.

\smallskip

{\bf 1.}  Let $\e=\sqrt{e^h}/h$ so $\e\to0$ as $h\to0$, by assumption \eqref{scaling-intro}.
Therefore, by Theorem \ref{th_nth-intro} there exists a sequence
$w_\e:\bar S\longrightarrow \mathbb{R}^3$, equibounded in
$\mathcal{C}^{1,1}(\bar S)$, such that for all small $h>0$:
\begin{equation}\label{ex-iso}
u_\e = \mathrm{id} +\e V + \e^2w_\e
\end{equation}
is a (generalized) $N$th order infinitesimal isometry.
Note that by (\ref{ass33}) we have:
$$ \frac {\e^{N+1}}{\sqrt{e^h}} = \frac{(\sqrt{e^h})^{N}}{h^{N+1}}= o (h^{N+1})/h^{N+1} \to 0, $$
hence $\e^{N+1}= o(\sqrt{e^h})$. We may thus replace $\mathcal{O}(\e^{N+1})$ with
$o(\sqrt{e^h})$.

For every $p\in S$, let $\vec n_\e(p)$ denote the unit normal vector
to $u_\e(S)$ at the point $u_\e(p)$. Clearly,
$\vec n_\e\in {\mathcal C}^{0,1}(\bar S,{\mathbb R}^3)$, while by \eqref{ex-iso}:
\begin{equation}\label{ne-exp}
\vec n_\e = \frac{\partial_{\tau_1}u_\e \times \partial_{\tau_2}u_\e}
{|\partial_{\tau_1}u_\e \times \partial_{\tau_2}u_\e|} = \vec n +\e A\vec n + \mathcal{O}(\e^2).
\end{equation}
Here $\tau_1,\tau_2\in T_pS$ are such that $\vec n=\tau_1\times\tau_2$.
Note that since $N>1$ and $u_\e$ is a (generalized) $N$th order isometry, we have
$ |\partial_{\tau_i} u_\e|^2 = 1 +  \mathcal{O}(\e^{3})$ and  $|\partial_{\tau_1} u_\e
\cdot \partial_{\tau_2} u_\e| =  \mathcal{O}(\e^{3}) $, which implies that:
$$ |\partial_{\tau_1}u_\e \times \partial_{\tau_2}u_\e| = 1+  \mathcal{O}(\e^{3}).$$
Using now the Jacobi identity for vector product and the fact that $A\in so(3)$, we
arrive at (\ref{ne-exp}).

Here we introduce the recovery sequence $u^h$ as required by the statement of the theorem.
Note that the following suggestion for $u^h$ is in accordance with the one
used in \cite{FJMM_cr} in the framework of the purely nonlinear bending theory
for shells, corresponding to the scaling regime
$\beta=2$. Also, a comparison with the similar proof 
in \cite{lemopa3} for convex shells, with which much of the following calculations overlap, 
is elucidating. Indeed, the main difference here is that instead of an exact isometry of 
the given shell, we make use of an $N$th order isometry $u_\varepsilon$. 
Consider the sequence of 
deformations $u^h\in W^{1,2}(S^h,{\mathbb R}^3)$ defined by:
\begin{equation}\label{rec_seq}
u^h(p+t\vec n) = u_\varepsilon(p) + t\vec n_\varepsilon(p)
+ \frac{t^2}{2}\varepsilon d^h(p),
\end{equation} where $\varepsilon$ depends on $h$ as above. The vector field $d^h\in W^{1,\infty}(S,\mathbb{R}^3)$ is taken so that:
\begin{equation}\label{nd01h}
\lim_{h\to 0} h^{1/2} \|d^h\|_{W^{1,\infty}(S)} = 0,
\end{equation}
and:
\begin{equation}\label{warp}
\lim_{h\to 0} d^h = 2c\left(p,{\rm sym}(\nabla(A\vec n) - A\Pi)_{tan}\right) \quad
\mbox{ in } L^\infty(S),
\end{equation}
where $c(p,F_{tan})$ denotes the unique vector satisfying
${\mathcal Q}_2(p, F_{tan})={\mathcal Q}_3(F_{tan} +c\otimes \vec n(p)+\vec
n(p)\otimes c)$ (see \cite[Section~6]{lemopa1}). We observe that, as
$V\in{\mathcal C}^{1,1}(\bar S, {\mathbb R}^3)$ and $c$ depends linearly
on its second argument, the vector field:
\begin{equation}\label{def-zeta}
\zeta(p)=c(p, {\rm sym}(\nabla(A\vec n) - A\Pi)_{tan})
\end{equation}
belongs to $L^\infty(S,\mathbb{R}^3)$.
Properties (i) and (ii) now easily follow from the uniform bound on $w_\e$
and the normalization \eqref{nd01h}.

\smallskip

{\bf 2.}
To prove (iii) it is convenient to perform a change of variables in
the energy $E^h(u^h)$, so to express it in terms of the scaled deformation $y^h$.
By a straightforward calculation:
\begin{equation}\label{int-nv}
\frac{1}{e^h} E^h(u^h) = \frac{1}{e^h}\int_S \fint_{-h_0/2}^{h_0/2}
W(\nabla_h y^h(p+t\vec n))\det[\mbox{Id}+th/h_0\Pi(p)]~\mbox{d}t\mbox{d}p,
\end{equation}
where $\nabla_h y^h(p+t\vec n)=\nabla u^h(p+th/h_0\vec n)$.
We also have:
\begin{equation}\label{form2}
\begin{split}
\nabla_h y^h(p + t\vec n) \vec n(p) & = \frac{h_0}{h} \partial_{\vec n}y^h(p+t\vec n)
= \vec n_\e(p) + th/h_0\e d^h(p),\\
\nabla_h y^h(p + t\vec n)\tau & = \nabla y^h(p+t\vec n)\cdot (\mbox{Id} +
t\Pi(p)) (\mbox{Id} + th/h_0\Pi(p))^{-1}\tau \\
& = \Big(\nabla u_\e(p) +th/h_0\nabla\vec n_\e(p) + \frac{t^2}{2h_0^2}h^2\e
\nabla d^h(p)\Big)(\mbox{Id} + th/h_0\Pi(p))^{-1}\tau,
\end{split}
\end{equation}
for all $p\in S$ and $\tau\in T_pS$.

From \eqref{ex-iso}, \eqref{ne-exp} and \eqref{nd01h} it follows that
$\|\nabla_h y^h-\mbox{Id}\|_{L^\infty(S^{h_0})}\to 0$ as $h\to 0$.
By polar decomposition theorem, $\nabla_h y^h$ is a product of a proper rotation and the well defined
square root $\sqrt{(\nabla_h y^h)^T\nabla_h y^h}$. By frame indifference of $W$ we
deduce that:
$$ W(\nabla_h y^h) = W\left(\sqrt{(\nabla_h y^h)^T\nabla_h y^h}\right)
= W\left(\mbox{Id} + \frac{1}{2} K^h + \mathcal{O}(|K^h|^2)\right), $$
where the last equality is obtained by the Taylor expansion, with:
$$ K^h =  (\nabla_h y^h)^T\nabla_h y^h - \mbox{Id}. $$
As $\|K^h\|_{L^\infty(S^{h_0})}$ is infinitesimal as $h\to 0$, we can expand
$W$ around $\mbox{Id}$, using the formula:
\begin{equation*}
\displaystyle W(\mbox{Id} + K) = \frac 12 D^2W(\mbox{Id})(K,K) + \int_0^1 (1-s) [D^2W(\mbox{Id}+sK) - D^2W(\mbox{Id})] (K,K) \mbox{d}s, 
\end{equation*} and obtain, in view of using the assumption 
that $W$ is ${\mathcal C}^2$ in a neighborhood of identity:
\begin{equation}\label{Wdopp}
\frac{1}{e^h} W(\nabla_h y^h) = \frac{1}{2} \mathcal{Q}_3\left(\frac{1}{2\sqrt{e^h}} K^h +
\frac{1}{\sqrt{e^h}}\mathcal{O}(|K^h|^2)\right) +
\frac{1}{{e^h}}o(|K^h|^2).
\end{equation}
Using \eqref{form2} we now calculate $K^h$. We first consider the tangential
minor of $K^h$, as usual conceived as a symmetric bilinear form on $S$:
\begin{equation*}
\begin{split}
K^h_{tan}(p + t\vec n) & =  (\mbox{Id} + th/h_0\Pi)^{-1}\Big[\mbox{Id} + \mathcal{O}(\e^{N+1})
+2 th/h_0\, {\rm sym}((\nabla u_\e)^T\nabla\vec n_\e)\\
& \qquad + t^2 h^2/h_0^2(\nabla \vec n_\e)^T\nabla\vec n_\e
+ o(\sqrt{e^h})\Big](\mbox{Id} + th/h_0\Pi)^{-1} - \mbox{Id}\\
& = (\mbox{Id} + th/h_0\Pi)^{-1}\Big[
2 th/h_0\, {\rm sym}((\nabla u_\e)^T\nabla\vec n_\e) -2 th/h_0\Pi\\
& \qquad + t^2 h^2/h_0^2(\nabla \vec n_\e)^T\nabla\vec n_\e - t^2 h^2/h_0^2\Pi^2
\Big](\mbox{Id} + th/h_0\Pi)^{-1} + o(\sqrt{e^h}),
\end{split}
\end{equation*}
where we used the fact that $u_\e$ is a generalized $N$th order
infinitesimal isometry to see that
$(\nabla u_\e)^T\nabla u_\e = \mbox{Id} +
\mathcal{O}(\e^{N+1})= \mbox{Id} + o(\sqrt{e^h})$, and the identity:
$$F_1^{-1}FF_1^{-1}-\mbox{Id} = F_1^{-1}(F-F_1^2)F_1^{-1}.$$
By \eqref{ex-iso} and \eqref{ne-exp} we also deduce:
\begin{equation*}
\begin{split}
{\rm sym} ((\nabla u_\e)^T\nabla\vec n_\e) &
= \Pi +\e\, {\rm sym}(\nabla(A\vec n)-A\Pi) +{\mathcal O}(\e^2), \\
(\nabla \vec n_\e)^T\nabla\vec n_\e & = \Pi^2+{\mathcal O}(\e).
\end{split}
\end{equation*}
Combining these two identities with the expression of $K^h_{tan}$ found above,
we conclude that:
\begin{equation*}
K^h_{tan}(p + t\vec n) =  \sqrt{e^h}(\mbox{Id} + th/h_0\Pi)^{-1}\Big[
2 t/h_0\, {\rm sym}(\nabla(A\vec n)-A\Pi)\Big](\mbox{Id} + th/h_0\Pi)^{-1} + o(\sqrt{e^h}).
\end{equation*}
Now, as $|\vec n_\e|=1$, the normal minor of $K^h$ is calculated by means of
(\ref{form2}) as:
\begin{equation*}
\vec n^T K^h(p+t\vec n)\vec n = |(\nabla_h y^h)\vec n|^2 - 1 =
2th/h_0\e d^h \cdot \vec n_\e+ o(\sqrt{e^h})
= 2t/h_0 \sqrt{e^h} d^h \cdot \vec n +  o(\sqrt{e^h}).
\end{equation*}
The remaining coefficients of the symmetric matrix $K^h(p+ t\vec n)$ are,
for $\tau\in T_x S$:
\begin{equation*}
\begin{split}
 \tau^T K^h(p+t\vec n)\vec n & = (\vec n_\e +th/h_0\e d^h)^T
\Big(\nabla u_\e+th/h_0\nabla \vec n_\e+
\frac{t^2}{2h_0^2}h^2\e\nabla d^h\Big)(\mbox{Id} + th/h_0\Pi)^{-1}\tau\\
& = t/h_0\sqrt{e^h}(d^h)^T\nabla u_\e(\mbox{Id} + th/h_0\Pi)^{-1}\tau + o(\sqrt{e^h}),
\end{split}
\end{equation*}
where we have used that $\vec n_\e^T\nabla\vec n_\e= \vec n_\e^T\nabla u_\e= 0$.

\smallskip

{\bf 3.} From the previous computations we finally deduce, with some abuse of notation, that:
\begin{equation}\label{Kconv}
\lim_{h\to 0} \frac{1}{2\sqrt{e^h}} K^h = \frac{t}{h_0}K(p)_{tan}
+ \frac{t}{h_0}(\zeta\otimes \vec n + \vec n\otimes \zeta)
\quad \mbox{ in } L^\infty(S^{h_0}),
\end{equation}
where the vector field $\zeta$ is defined in \eqref{def-zeta} and
the symmetric bilinear form $K_{tan}\in L^\infty(S)$ is:
\begin{equation}\label{Kdef}
K(p)_{tan} = {\rm sym} (\nabla(A\vec n) - A\Pi)_{tan}.
\end{equation}
Using \eqref{int-nv}, \eqref{Wdopp}, \eqref{Kconv} and the dominated
convergence theorem, we obtain:
\begin{equation*}
\begin{split}
\lim_{h\to 0} \frac{1}{e^h} E^h(u^h) & =
\lim_{h\to 0} \frac{1}{e^h} \int_S \fint_{-h_0/2}^{h_0/2} W(\nabla_h y^h)
\det (\mbox{Id} + th/h_0\Pi)~\mbox{d}t\mbox{d}p\\
& = \frac{1}{2} \int_S \fint_{-h_0/2}^{h_0/2} \mathcal{Q}_3 \Big(\frac{t}{h_0}K(p)_{tan}
+ \frac{t}{h_0}(\zeta\otimes \vec n + \vec n\otimes \zeta)\Big)~\mbox{d}t\mbox{d}p \\
& = \frac{1}{2}\int_S \fint_{-h_0/2}^{h_0/2}\frac{t^2}{h_0^2}\mathcal{Q}_2
\big(p,{\rm sym} (\nabla(A\vec n) - A\Pi)_{tan}\big)~\mbox{d}t\mbox{d}p,
\end{split}
\end{equation*}
the last equality following in view of \eqref{def-zeta} and
\eqref{Kdef}. Property (iii) now follows, upon integration in $t$
in the last integral above.
\end{proof}

\end{document}